\documentclass[a4paper,10pt]{amsart}

\usepackage{amssymb}
\usepackage{amsmath,amscd}
\usepackage{stmaryrd}
\usepackage[all]{xy}
\usepackage{epsf}
\usepackage{enumerate}
\usepackage{comment}
\usepackage{combelow}

\newcommand{\R}         {{\mathbb{R}}}

\newcommand{\Log}         {{\mathrm{Log}}}

\newcommand{\Poiss}         {{\mathrm{Poiss}}}
\newcommand{\Diff}         {{\mathrm{Diff}}}
\newcommand{\id}         {{\mathrm{id}}}

\newcommand{\rmap}{\longrightarrow}
\newcommand{\diffto}{\xrightarrow{\raisebox{-0.2 em}[0pt][0pt]{\smash{\ensuremath{\sim}}}}}

\newenvironment{remark}[1][Remark]{\begin{trivlist}
\item[\hskip \labelsep {\bfseries #1}]}{\end{trivlist}}

\newtheorem{theorem}{Theorem}
\newtheorem{lemma}{Lemma}
\newtheorem{definition}{Definition}
\newtheorem{example}{Example}
\newtheorem{proposition}{Proposition}
\newtheorem{corollary}{Corollary}

\title[Deformations of log symplectic structures]{Deformations of log symplectic structures}

\author{Ioan M\u{a}rcu\cb{t}}
\address{Depart. of Math., Univ. of Illinois at Urbana-Champaign, Urbana IL 61801, USA}
\email{marcut@illinois.edu}
\author{Boris Osorno Torres}
\address{Depart. of Math., Utrecht University, 3508 TA, Utrecht, The Netherlands}
\email{b.osornotorres@uu.nl}

\begin{document}

\begin{abstract}
We describe the space of Poisson bivectors near a log symplectic structure up to small diffeomorphisms.
\end{abstract}

\maketitle

\section{Introduction}

A Poisson structure on a smooth manifold $M$ is a Lie bracket on the space of functions $C^{\infty}(M)$ that satisfies the Leibniz rule
\[\{f,gh\}=\{f,g\}h+g\{f,h\},\]
for all $f,g,h\in C^{\infty}(M)$. Equivalently, a Poisson structure is given by a bivector $\pi\in \Gamma(\wedge^2
TM)$ that satisfies the equation $[\pi,\pi]=0$, where $[\cdot,\cdot]$ denotes the Schouten bracket. The bivector and
the bracket are related by
\[\{f,g\}=\pi(df,dg).\]
From this point of view, a symplectic structure is the same as a non-degenerate Poisson structure. The corresponding
two-form is given by $\omega=\pi^{-1}$, and non-degeneracy is equivalent to $\wedge^n \pi$ being nowhere zero,
where $\dim M=2n$. Generalizing this situation, a Poisson structure $\pi$ on a manifold $M$ of dimension $2n$ is
called \textit{log symplectic} if the map
\[\wedge^n \pi: M \rmap \bigwedge\nolimits^{\!2n} TM, \ \ x \mapsto \wedge^n \pi(x)\]
is transverse to the zero section. These structures are very close to being symplectic, as they degenerate along a
codimension-one submanifold, and therefore, many of the results of symplectic geometry can be extended to this
framework. Note that in this paper, log symplectic manifolds are not required to be orientable.

Initially, log symplectic manifolds were considered on manifolds with boundary in \cite{Nest.1996}, in the context of deformation quantization, and were studied using the tools of b-geometry developed in \cite{Melrose.1993}. On compact oriented surfaces (without boundary), these structures were fully classified in \cite{Radko.2002}. More recently, the systematic study of log symplectic structures from \cite{Guillemin.Eva.2012} attracted a lot of attention and interest. As a result, several new papers appeared studying different facets of the geometry of log symplectic manifolds \cite{Gualtieri.2012,MO,Cavalcanti.2013,Scott,PDE}.

\vspace{0.2cm}

In this paper we give an explicit description of the space of Poisson structures $C^1$-close to a given log symplectic structure modulo small diffeomorphisms. In the general case, the problem of describing deformations of Poisson structures is poorly understood, as the nature of the deformation space depends to a great extent on the structure under consideration, with some structures showing rigidity phenomena \cite{Conn} while for others the deformation space is infinite dimensional \cite{Marcut.Lie.Sphere}. The main difficulty (in comparison for example with complex geometry) arises from the non ellipticity of the complex describing the infinitesimal deformations, called the Poisson complex. For log symplectic manifolds, the Poisson complex is elliptic outside a codimension-one hypersurface, and this already indicates that the problem is more tractable in this case. A similar behavior is displayed by $b^k$-symplectic structures (introduced in \cite{Scott}), and the authors expect that deformations of these structures can be studied in a similar fashion to obtain an analogue description of the space of $C^k$-deformations.

\vspace{0.2cm}

The geometry of a Poisson manifold can be highly nontrivial, making even local results very hard. The global picture
(e.g.\ of the distribution of symplectic leaves) can also be rather wild and unexpected.
Log symplectic structures have a very specific geometry,
which is more accessible, but still nontrivial enough to allow for interesting results. \emph{The singular locus} of a
log symplectic structure $(M,\pi)$ is the codimension-one Poisson submanifold $Z:=(\wedge^n\pi)^{-1}(0)$. More insight into the geometry of a log symplectic structure is gained when looking at the inverse of $\pi$. For this,
fix $E\to M$ a tubular neighborhood of $Z$ in $M$. For a metric $x\mapsto |x|$ on $E$, fix a function
$\lambda:M\backslash Z\to (0,\infty)$, satisfying $\lambda(x):=|x|$, for $x\in E$ with $|x|\leq 1/2$, and
$\lambda\equiv1$ on $M\backslash \{x\in E: |x|<1\}$. Using this function, the inverse $\pi$ can be written on $M\setminus Z$ in the
following form
\begin{equation}\label{inverse}
\pi^{-1}=\alpha+d\log (\lambda)\wedge p^*(\theta),
\end{equation}
where $\alpha$ is a closed two-form on $M$, $\theta$ is a closed one-form on $M$ and $p:E\to Z$ is the projection. The
pair consisting of $\eta:=\alpha|_{Z}\in \Omega^2(Z)$ and $\theta\in \Omega^1(Z)$ is a \emph{cosymplectic
structure} on $Z$, i.e.\ $\eta$ and $\theta$ are closed and $\theta\wedge \eta^{n-1}$ is a volume form. The
cosymplectic structure determines a pair $(\pi_Z,V)$, where $\pi_Z$ is a Poisson structure on $Z$ whose symplectic foliation is given by the kernel of
$\theta$ endowed with the pullback of $\eta$, and $V$ is the vector field such that $\iota_V\eta=0$ and $\theta(V)=1$. Moreover, this gives a one-to-one correspondence between cosymplectic structures and regular corank-one Poisson structures endowed with a transverse Poisson vector field (see e.g. \cite{Guillemin.Eva.2011}).

In this paper we show that, up to diffeomorphism, there are only two ways to deform log symplectic structures on compact manifolds. These are described as follows:
\begin{enumerate}[i.]
\item The first type of deformation is the \emph{gauge transformation} by a (small enough) closed two-form
$\varpi\in \Omega^2(M)$. This is a general operation in Poisson geometry, which transforms $\pi$ by adding the
restriction of $\varpi$ to the symplectic form on each leaf of $\pi$.
\item The second type of deformation is specific to log symplectic structures and it transforms $\pi$ locally around $Z$. For a (small enough)
closed one-form $\gamma\in \Omega^1(Z)$, the transformed Poisson structure is also log symplectic with singular locus $Z$,
but with foliation on $Z$ given by the kernel of $\theta+\gamma$.
\end{enumerate}
These transformations can be described algebraically in terms of their inverses on $M\backslash Z$; the result of
transforming $\pi$ by the pair $(\varpi,\gamma)\in \Omega^2(M)\times\Omega^1(Z)$ is given by
\begin{equation}\label{tranformed}
\left(\pi^{\varpi}_{\gamma}\right)^{-1}:=\pi^{-1}+\varpi+d \log(\lambda)\wedge p^*(\gamma).
\end{equation}

The cosymplectic structure on $Z$ induced by $\pi^{\varpi}_{\gamma}$ is the pair
$(\eta+\varpi|_{Z},\theta+\gamma)$.

In fact, these two types of deformations cover all Poisson structures near a log symplectic structure.

\begin{theorem}\label{deformations1}
Let $(M,\pi,Z)$ be a compact log symplectic manifold. Consider $\varpi_1,\ldots,\varpi_l$ closed two-forms on $M$
and $\gamma_1,\ldots,\gamma_k$ closed one-forms on $Z$ such that their cohomology classes form a basis of
$H^2(M)$ and of $H^1(Z)$ respectively. For $\epsilon\in\mathbb{R}^l$ and $\delta\in \mathbb{R}^k$, denote by
$\varpi_{\epsilon}:=\sum_{i=1}^l\epsilon_i\varpi_i$ and $\gamma_{\delta}:=\sum_{i=1}^k\delta_i\gamma_i$. Then,
\begin{enumerate}[(a)]
\item for all small enough $\epsilon\in \mathbb{R}^l$ and $\delta\in \mathbb{R}^k$, we have that
$\pi^{\varpi_{\epsilon}}_{\gamma_{\delta}}$ (given by (\ref{tranformed})) is a log symplectic structure;
\item there is a $C^1$-open neighborhood $\mathcal{U}\subset \Gamma(\wedge^2 TM)$ around $\pi$, such that every Poisson structure
$\pi'\in \mathcal{U}$ is isomorphic to $\pi^{\varpi_{\epsilon}}_{\gamma_{\delta}}$ for some vectors
$\epsilon\in\mathbb{R}^l$, $\delta\in \mathbb{R}^k$;
\item there is a $C^1$-neighborhood ${\mathcal{D}}\subset\Diff(M)$ around $\id_M$ such that for
$\varphi\in {\mathcal{D}}$, the equality
$\varphi_*(\pi^{\varpi_{\epsilon}}_{\gamma_{\delta}})=\pi^{\varpi_{\epsilon'}}_{\gamma_{\delta'}}$ implies
${\epsilon}={\epsilon}'$ and ${\delta}={\delta}'$.
\end{enumerate}
\end{theorem}


The natural setting for studying log symplectic structures is the formalism of \textit{b-geometry} developed in \cite{Melrose.1993}. The cohomology of interest in this formalism is the b-cohomology, that can be explicitly computed in terms of the de Rham cohomology.
Theorem \ref{deformations1} states that the cohomology classes $([\alpha],[\theta])$ from (\ref{inverse}) give a parametrization of the space of Poisson structures near $\pi$ up to small diffeomorphisms. In terms of b-cohomology, this says that Poisson structures near a log symplectic manifold $(M,\pi,Z)$ are parameterized by the second b-cohomology group $H^2_Z(M)$. This is in perfect analogy with the situation in symplectic geometry, where deformations of symplectic forms are classified by the second de Rham cohomology group.

For $M$ a compact oriented surface, Theorem \ref{deformations1} follows from the classification in \cite{Radko.2002}: writing $Z$ as a union of circles $Z_1,\ldots, Z_k$ and decomposing $[\theta]=\sum_i [\theta_i]$, with $[\theta_i]\in H^1(Z_i)$, the classes $[\theta_i]$ correspond to the periods of the modular vector field along the curves $Z_i$ and $[\alpha]$ corresponds to the \textit{generalized Liouville volume} of $(M,\pi)$.

Recall that a log symplectic structure $(M,\pi,Z)$ is called \emph{proper} \cite{Gualtieri.2012}, if the foliation on $Z$ is
given by a submersion to $S^1$. Now this is equivalent to the cohomology class $[\theta]$ being a real multiple of an
integral class (see \cite{Tischler}). Since such one-forms are dense in the space of all closed one-forms, Theorem \ref{deformations1} implies that the
proper log symplectic structures form a dense set in the space of all log symplectic structures, a result which appeared also in
\cite{Cavalcanti.2013}.

\vspace{0.2cm}

The structure of the paper is the following: In section 2 we present the language of b-geometry. In particular, we discuss the
functoriality of the Mazzeo-Melrose decomposition of b-cohomology (Corollary \ref{functorial}), which implies that the classes
$[\alpha]$, $[\theta]$ associated to $\pi$ are canonical. In the orientable case, this functoriality result appeared in \cite{Scott}. The b-geometric version of the Moser lemma from \cite{Nest.1996}, which plays a fundamental role in our proof, is stated in Lemma \ref{Moser} in the more general setting of one-,two- and top-forms. In the case of top forms, we obtain a simple proof (Corollary \ref{Nambu}) of the classification of generic multi-vector fields of top degree from \cite{David}. In section 3 we compute the Poisson cohomology of log symplectic structures (Proposition \ref{cohomology}) and relate this to Theorem \ref{deformations1}. In section 4 we begin the proof of Theorem \ref{deformations1} by proving that any Poisson structure close to $\pi$ is Poisson diffeomorphic to a log symplectic structure also close to $\pi$ with the same singular locus. In section 5 we finish the proof of Theorem \ref{deformations1}. In section 6 we present examples which illustrate the phenomena that can occur when deforming log symplectic structures. In the appendix, we include the proof of Lemma \ref{functorial1}.

\medskip

\noindent \textbf{Acknowledgments.} We would like to thank Marius Crainic for his comments on the first drafts of this paper. We are also grateful to David Mart\'inez Torres for insightful discussions, which helped us understand better the language of ``b-geometry''. We also thank the referee for useful comments. The first author was supported by the ERC Starting Grant no. 279729 and the second by the NWO VIDI project ``Poisson Topology'' no. 639.032.712.

\section{The language of b-geometry}\label{b-geom}

A very effective tool to describe and handle log symplectic structures is using an adaptation of the b-geometry of
manifolds with boundary developed by Melrose in \cite{Melrose.1993} (see also \cite{Guillemin.Eva.2012}).

We consider pairs $(M,Z)$, where $M$ is a smooth manifold (without boundary) and $Z$ is a distinguished
codimension-one closed embedded submanifold of $M$. Such a pair is called a \textit{b-manifold}. A \textit{b-map} is a smooth
map $\varphi:(M_1,Z_1)\to (M_2,Z_2)$ between b-manifolds that is transverse to $Z_2$ and satisfies
$\varphi^{-1}(Z_2)=Z_1$.

Associated to a pair $(M,Z)$ there is a natural \textit{b-tangent bundle} denoted by $T_ZM\to M$, whose sections are
vector fields on $M$ that are tangent to $Z$. Let $E\to M$ be a tubular neighborhood of $Z$ in $M$, and denote by
$\xi$ the corresponding  Euler vector field on $E$; i.e.\ the flow of $\xi$ is fiberwise multiplication with $e^t$.
Regarding $\xi$ as a section of $T_ZE$, we see that $\xi$ is nowhere vanishing, and moreover, that
\begin{equation}\label{xizi} \xi_Z:=\xi|_Z\in \Gamma(T_ZM|_Z)\end{equation}
is independent of the tubular neighborhood.

\subsection{B-cohomology}
There is also a notion of \textit{b-de Rham cohomology}, denoted by $H^{\bullet}_Z(M)$, which is computed on the
complex of \textit{b-differential forms} $\Omega^{\bullet}_Z(M):=\Gamma(\wedge^{\bullet}T^*_ZM)$ (i.e.\ on the space of
multi-linear forms on $T_ZM$). The differential is determined by the fact that the restriction map
$\Omega^{\bullet}_Z(M)\to \Omega^{\bullet}(M\backslash Z)$ is a chain map. It is well-known that this complex fits
in a (canonical) short exact sequence of complexes:
\begin{equation}\label{Exact}
0\rmap\Omega^{\bullet}(M)\rmap \Omega^{\bullet}_Z(M)\stackrel{\iota_{\xi_Z}}{\rmap} \Omega^{\bullet-1}(Z)\rmap 0,
\end{equation}
where $\xi_Z$ is defined by (\ref{xizi}). This sequence splits. To see this, fix $E$ a tubular neighborhood of $Z$, and
consider a function $\lambda:M\backslash Z\to (0,\infty)$ with the properties from the Introduction: $\lambda(x):=|x|$,
for $x\in E$ with $|x|\leq 1/2$, and $\lambda\equiv1$ on $M\backslash \{x\in E: |x|<1\}$. We call such a function
\emph{a distance function adapted to} $E$. The one-form $d\log(\lambda)$ extends to a closed one-form on $T_ZM$
which is supported in $E$. A splitting of (\ref{Exact}) commuting with the differentials is given by the map
\begin{equation}\label{splitting}
\sigma: \Omega^{\bullet-1}(Z)\rmap \Omega^{\bullet}_Z(M),\ \ \ \omega\mapsto d\log(\lambda)\wedge p^*(\omega),
\end{equation}
where $p:E\to Z$ is the projection. This implies the Mazzeo-Melrose theorem \cite{Melrose.1993}, i.e.\ we have the
following decomposition for b-cohomology:
\begin{equation}\label{Mazzeo}
H^{\bullet}_Z(M)= H^{\bullet}(M)\oplus \sigma(H^{\bullet-1}(Z))\cong H^{\bullet}(M)\oplus H^{\bullet-1}(Z).
\end{equation}

Note that the image of $1\in H^0(Z)$ under $\sigma$ is
\[\sigma(1)=[d\log(\lambda)].\]
Now, if $\lambda'$ is another distance function (associated to a second tubular neighborhood), it is easy to see that
there is a smooth function $f\in C^{\infty}(M)$ such that $\lambda'=e^f\lambda$. Hence,
\[d\log(\lambda')=d\log(\lambda)+df,\]
and therefore, the class $[d\log(\lambda)]\in H^1_Z(M)$ is independent of the choice of $\lambda$.

This example motivates the results below. The case when $M$ is orientable of corollary \ref{functorial} appeared in \cite{Scott} in the more general setting of
$b^k$-forms. Here orientability is not assumed. The proof of the lemma is presented in the appendix so as not to disrupt the flow of the paper.

\begin{lemma}\label{functorial1}
The elements in $\sigma(H^{k-1}(Z))\subset H^k_Z(M)$ are characterized by the following property:
if $\omega\in \Omega^k_Z(M)$ is a closed b-form, then $[\omega]\in H^k_Z(M)$ belongs to $\sigma(H^{k-1}(Z))$ if and only if for every b-map $i:(N,W)\to (M,Z)$ from a $k$-dimensional compact oriented b-manifold $(N,W)$, the following holds:
\[\int_{N}i^*(\omega)=0,\]
where the integral is the regularized volume defined in \cite{David}.
\end{lemma}

\begin{corollary}[(see also \cite{Scott})]\label{functorial}
The decomposition (\ref{Mazzeo}) is functorial in the b-category.
\end{corollary}

\begin{proof}
Let $\varphi:(M_1,Z_1)\to (M_2,Z_2)$ be a b-map. We have to show that the induced map
\[\varphi^*:H^{\bullet}_{Z_2}(M_2)\rmap H^{\bullet}_{Z_1}(M_1)\]
maps $\sigma(H^{\bullet-1}(Z_2))$ into $\sigma(H^{\bullet-1}(Z_1))$, where $\sigma$ is the map from (\ref{splitting}). Let $[\omega]\in \sigma(H^{k-1}(Z_2))$ be represented by a closed b-forms $\omega\in \Omega^k_{Z_2}(M_2)$. Let $i:(N,W)\to (M_1,Z_1)$ be a b-map from a compact oriented $k$-dimensional b-manifolds $(N,W)$. Then $\varphi\circ i:(N,W)\to (M_2,Z_2)$ is also a b-map; therefore, by the lemma,
\[\int_{N}i^*(\varphi^*(\omega))=\int_{N}(\varphi\circ i)^*(\omega)=0.\]
Again, the lemma implies that $[\varphi^*(\omega)]=\varphi^*[\omega]\in \sigma(H^{k-1}(Z_2))$. This finishes the proof.
\end{proof}

\subsection{Log symplectic structures}

The framework of b-geometry allows us to regard log symplectic structures as ``symplectic'' structures on the b-tangent bundle
$T_ZM$. To see this, note first that a log symplectic structure on $M$ with singular locus $Z$ is the same as a
nondegenerate section $\pi\in \Gamma(\wedge^2 T_ZM)$ satisfying $[\pi,\pi]=0$. This is equivalent to having a
nondegenerate b-two-form $\omega\in \Gamma(\wedge^2 T^*_ZM)$ that is closed, $d\omega=0$; the two being
related by $\pi^{-1}=\omega$. Using a tubular neighborhood of $Z$ in $M$ and a distance function $\lambda$ adapted
to $E$, we can decompose
\[\omega=\alpha+d\log(\lambda)\wedge p^*(\theta),\]
where $\alpha$ is a closed two-form on $M$ and $\theta$ is a closed one-form on $Z$. Note that the image of any
closed b-one-form under the isomorphism $\pi^{\sharp}:T^*_ZM\diffto T_ZM$ is a Poisson b-vector field. In particular,
$X:=-\pi^{\sharp}(d\log(\lambda))$ represents the modular class of $\pi$ (see \cite{Guillemin.Eva.2012}) and
$V:=X|_{Z}$ is a Poisson vector field on $(Z,\pi_{Z})$ transverse to the symplectic leaves. As mentioned in the Introduction, the pair $(\pi_Z,V)$
corresponds to the cosymplectic structure $(\eta,\theta)$ on $Z$, where $\eta:=\alpha|_Z$. Note that $\theta$ does not
depend on the tubular neighborhood and on $\lambda$, since it is the image of $\omega$ under the map $\iota_{\xi_Z}$ from (\ref{Exact}). However, $\eta$ does depend, but in a rather mild fashion, as it changes only by exact two-forms of the type
$d(f\theta)$, with $f\in C^{\infty}(Z)$. Likewise, $V$ is determined up to Hamiltonian vector fields.

Now we reformulate the main result in the b-language. If $M$ is compact and $\mu\in
\Omega^2_Z(M)$ is a $C^0$-small closed b-two-form, then  $\omega+\mu$ is still non-degenerate, therefore $\pi^{\mu}:=(\omega+\mu)^{-1}$ is
a small deformation of $\pi$. This remark implies part (a) of Theorem \ref{deformations1}. Part (b) says that every
Poisson structure close to $\pi$ is isomorphic to one of this form, and moreover, whenever $[\mu]=[\mu']\in
H^2_Z(M)$, we have that $\pi^{\mu}$ and $\pi^{\mu'}$ are isomorphic. Conversely, part (c) says that there is an
open neighborhood $\mathcal{D}$ around $\id_M$ in the space of diffeomorphisms of $M$, such that if $[\mu]\neq [\mu']$, then
$\pi^{\mu}$ and $\pi^{\mu'}$ are not related by an element in $\mathcal{D}$. In other words, Poisson structures around $\pi$ are parameterized up to small diffeomorphisms by an open neighborhood in $H^2_Z(M)$ around $[\omega]$; or equivalently, by an open neighborhood in $H^2(M)\oplus H^1(Z)$ around the pair $([\alpha],[\theta])$. Moreover, by Corollary \ref{functorial}, the pair $([\alpha],[\theta])$ is canonically associated to $\pi$.

\subsection{Non-degenerate b-forms}\label{b-Moser}

In the proof of Theorem \ref{deformations1} we use a version of the Moser trick for closed non-degenerate
b-two-forms. As in the classical case, this stability result holds for the following type of forms:

\begin{definition}
A b-form $\zeta\in \Omega^k_Z(M)$ of degree $k$ on $(M,Z)$ is called \emph{nondegenerate} if the following map is
surjective:
\[T_ZM\rmap\bigwedge\nolimits^{\!k-1}T^*_ZM,\ \ V\mapsto \iota_V\zeta.\]
\end{definition}

Comparing dimensions, we see that nondegenerate b-forms can exist only in degrees $1,2$ and $\dim(M)$.
Correspondingly, closed nondegenerate b-forms give rise to the following geometric objects:
\begin{enumerate}
\item[(1)] Let $\vartheta$ be a closed nondegenerate b-one-form. Nondegeneracy is equivalent to $\vartheta$ being nowhere
vanishing. Now $\iota_{\xi_Z}\vartheta$ is a closed 0-form on $Z$, thus on each connected component $Z_i$ of $Z$
it is a constant $c_i\in\mathbb{R}$. Geometrically, $\vartheta$ encodes a codimension one foliation $\mathcal{F}_{\vartheta}$ on $M$ which on
$M\backslash Z$ is given by the kernel of $\vartheta$, it is transverse to the components $Z_i$ for which $c_i=0$, and it contains the components $Z_i$ with $c_i\neq 0$ as leaves. To see that this defines indeed a smooth foliation, let $E_i\to M$ be tubular neighborhoods of $Z_i$, such that
$E_i\cap E_j=\emptyset$, for $i\neq j$, and let $\lambda_i$ be adapted distance functions. We decompose
$\vartheta$ as the locally finite sum
\[\vartheta=\theta+\sum_ic_i d\log(\lambda_i),\]
where $\theta$ is a closed one-form. Around a component with $c_i=0$, we have that $\vartheta=\theta$, and
nondegeneracy implies that $\theta|_{TZ_i}:TZ_i\to \mathbb{R}$ is nowhere vanishing; hence the foliation extends
to $Z_i$ by the kernel of $\theta$, and moreover, it is transverse to $Z_i$. If $Z_i$ is a component with $c_i\neq 0$,
then around $Z_i$ we can write $\vartheta=\theta+c_id\log(\lambda)$. On small open neighborhoods $U_i$ around points in $Z_i$,
we can write $\lambda_i=|t|$, where $t$ is a coordinate function with $Z_i\cap U_i=\{t=0\}$. Then the kernel of
$\vartheta$ is also the kernel of the 1-form
\[t/c_i\vartheta=dt+t/c_i\theta,\]
which shows that the foliation extends smoothly to $Z_i$, with $Z_i$ as a leaf.
\item[(2)] As discussed in the previous section, closed nondegenerate b-two-forms $\omega$ are the same as log symplectic
structure $\pi(=\omega^{-1})$ with singular locus $Z$.
\item[(3)] Let $\mu$ be a b-top-form on $(M^m,Z)$. Then $w:=\mu^{-1}\in
\Gamma(\wedge^{m}TM)$ is a multi-vector field of top degree on $M$, which intersects the zero-section of
$\wedge^{m}TM$ transversally at $Z$. These structures are called generic \emph{Nambu structures} of top
degree and were studied in \cite{David}. In Corollary \ref{Nambu} below, we show that the b-geometric Moser
argument implies the main result from \emph{loc.cit.}
\end{enumerate}

We give now the b-geometric Moser lemma for non-degenerate b-forms. The proof is the same as in the case for symplectic b-two-forms, which appeared first in \cite{Nest.1996}.

\begin{lemma}[(see \cite{Nest.1996})]\label{Moser}
Let $\zeta\in \Omega^k_Z(M)$ be a closed nondegenerate b-k-form on a compact b-manifold $(M,Z)$, where $k\in \{1,2,\dim (M)\}$. If $\zeta'\in
\Omega^k_Z(M)$ is a closed b-$k$-form, such that $(1-t)\zeta+t\zeta'$ is nondegenerate for all $t\in [0,1]$, and
\[[\zeta]=[\zeta']\in H^k_Z(M),\]
then there exists a b-diffeomorphism $\varphi:(M,Z)\diffto (M,Z)$, such that $\varphi^*(\zeta')=\zeta$.
\end{lemma}

As a consequence of this result, we obtain the following:
\begin{corollary}\label{Nambu}
Let $\mu,\mu'\in \Omega_Z^{\textrm{top}}(M)$ be nowhere vanishing b-top-forms on a compact b-manifold $(M,Z)$.
If $[\mu]=[\mu']\in H^{\textrm{top}}_Z(M)$, then there exists a b-diffeomorphism $\varphi$ of $(M,Z)$ such that
$\varphi^*(\mu')=\mu$.
\end{corollary}
\begin{proof}
Since $\wedge^{\mathrm{top}}T^*_ZM$ is a rank-one bundle, we can write $\mu'=f\mu$, for some nowhere
vanishing function $f\in C^{\infty}(M)$. Let us show that $f>0$. If $Z=\emptyset$, it follows that $M$ is orientable,
therefore the claim follows by integrating $\mu$ and $\mu'$ on $M$ with respect to the same orientation. If
$Z\neq\emptyset$, we have that $\iota_{\xi_Z}\mu$ and $\iota_{\xi_Z}\mu'$ are volume forms on $Z$ in the same
cohomology class, thus we can apply the same argument as before to conclude that $f|_Z>0$. Hence, $f>0$
everywhere. This implies that $(1-t)\mu+t\mu'=((1-t)+tf)\mu$ is nowhere vanishing for all $t\in [0,1]$ and the result
follows from the Moser Lemma.
\end{proof}

Using the decomposition from Corollary \ref{functorial}
\[H^{\textrm{top}}_Z(M)\cong H^{\textrm{top}}(M)\oplus H^{\textrm{top}}(Z),\]
we see that Corollary 1 extends the classification of generic Nambu structures of top degree from \cite{David} to the
case of non-orientable manifolds $M$. In the orientable case, fixing orientations on $M$ and on the components of $Z$,
we have that a generic Nambu structure $w$ of top degree with singular locus $Z$ is determined, up to orientation
preserving diffeomorphisms, by the regularized volume of $\mu:=w^{-1}$ and by the volumes of the components $Z_i$
of $Z$ computed with the aid of $\iota_{\xi_Z}\omega|_{Z_i}$. In the non-oriented case, it is determined entirely by the
volumes of the components.

\section{Poisson cohomology of log symplectic manifolds}\label{coho}

The Poisson cohomology $H^{\bullet}_{\pi}(M)$ of a Poisson manifold $(M,\pi)$ is computed by the complex
$\mathfrak{X}^{\bullet}(M):=\Gamma(\wedge^{\bullet}TM)$ of multi-vector fields on $M$ endowed with the
differential $d_{\pi}:=[\pi,\cdot]$, where $[\cdot,\cdot]$ denotes the Schouten bracket. The second Poisson cohomology
group $H^2_{\pi}(M)$ has the heuristical interpretation of being the ``tangent space'' at $\pi$ of the moduli-space of
all Poisson structures on $M$. In this section we show that for log symplectic structures this interpretation is in fact accurate.
By Theorem \ref{deformations1}, deformations of a compact log symplectic manifold $(M,\pi)$ with singular locus $Z$ are parameterized by $H^2_Z(M)$. At the cohomological level, this is expressed as follows:

\begin{proposition}\label{cohomology}
Let $(M,\pi)$ be a log symplectic manifold and $Z$ its singular locus. The Poisson cohomology of $(M,\pi)$ is
isomorphic to the b-cohomology of $(M,Z)$:
\begin{equation*}
H_{\pi}^{\bullet}(M) \cong H_Z^{\bullet}(M).
\end{equation*}
\end{proposition}

The proof uses the space $\mathfrak{X}^{\bullet}_Z(M)$
of multi-vector fields on $M$ tangent to $Z$, which can also be identified with the space of b-multi-vector fields on
$(M,Z)$, i.e.\ $\mathfrak{X}^{\bullet}_Z(M)=\Gamma(\wedge^{\bullet} T_ZM)$. Note that, since $Z$ is a Poisson
submanifold, $(\mathfrak{X}^{\bullet}_Z(M),d_{\pi})$ is a subcomplex of $(\mathfrak{X}^{\bullet}(M),d_{\pi})$. The
resulting cohomology, denoted by
$H^{\bullet}_{\pi,Z}(M)$, is called the ``b-Poisson-cohomology'' in \cite{Guillemin.Eva.2012}. In \textit{loc.cit.}\ it is shown that the log symplectic structure $\pi$ gives an isomorphism of complexes
\[\wedge^{\bullet}\pi^{\sharp}:(\Omega^{\bullet}_Z(M),d) \diffto (\mathfrak{X}^{\bullet}_Z(M),d_{\pi}),\]
and therefore $H^{\bullet}_Z(M)\cong H^{\bullet}_{\pi,Z}(M)$.
Hence the following lemma implies Proposition \ref{cohomology}.

\begin{lemma}\label{A_acyclic}
The inclusion $\mathfrak{X}^{\bullet}_Z(M)\subset \mathfrak{X}^{\bullet}(M)$ induces an isomorphism in
cohomology; i.e. the Poisson cohomology is isomorphic to the b-Poisson cohomology:
\[H^{\bullet}_{\pi}(M)\cong H^{\bullet}_{\pi,Z}(M).\]
\end{lemma}
\begin{proof}
We will construct linear maps $h: \mathfrak{X}^{\bullet}(M)\to  \mathfrak{X}^{\bullet-1}(M)$ that satisfy
\begin{equation}\label{homot}
\zeta(w):=w+d_{\pi}\circ h(w)+h\circ d_{\pi}(w)\in  \mathfrak{X}^{\bullet}_Z(M),\textrm{ for all }w\in  \mathfrak{X}^{\bullet}(M).
\end{equation}
This will imply the conclusion, since $\zeta$ induces a map in cohomology $\zeta:H^{\bullet}_{\pi}(M)\to
H^{\bullet}_{\pi,Z}(M)$, which is the inverse of the map induced by the inclusion.

Let $E\to M$ be a tubular neighborhood of $Z$ in $M$, with projection $p:E\to Z$. Let $\chi$ be a smooth function
supported in $E$, such that $\chi=1$ in a neighborhood of $Z$. Denote $\omega:=\pi^{-1}\in \Omega^2_Z(M)$ and
consider the canonical 1-form $\theta\in\Omega^1(Z)$ defining the foliation on $Z$. Now, define the operators $h$ by
\[h: \mathfrak{X}^{\bullet}(M)\rmap  \mathfrak{X}^{\bullet-1}(M), \ \ h(w):=\iota_{p^*(\theta)}\left(\chi w\right).\]
It suffices to check (\ref{homot}) locally around $Z$, where we have that $\chi=1$. First, note that, since
$p^*(\theta)$ is closed, a Poisson version of the Cartan formula holds:
\[\iota_{p^*(\theta)}\circ d_{\pi}+d_{\pi}\circ\iota_{p^*(\theta)}=L_{\pi^{\sharp}(p^*(\theta))}.\]
Thus, for (\ref{homot}) it suffices to check that $w+L_{\pi^{\sharp}(p^*(\theta))}w\in
\mathfrak{X}^{\bullet}_Z(M)$. Since $\pi$ is tangent to $Z$, $\xi:=\pi^{\sharp}(p^*(\theta))$ is a b-vector field.
Recall that $\theta=\omega^{\sharp}|_{Z}(\xi_Z)$, where $\xi_Z$ is the canonical section of $T_ZM|_Z$. Since
\[\omega^{\sharp}(\xi)|_{Z}=p^*(\theta)|_{Z}=\theta,\]
and $\omega$ is invertible, we have that $\xi|_{Z}=\xi_Z$. Now, an easy local computation shows that every b-vector
field $\xi$ extending $\xi_Z$ satisfies $w+[\xi,w]\in \mathfrak{X}^{\bullet}_Z(M)$. \end{proof}

\begin{remark}
Analyzing the proof of Lemma \ref{A_acyclic}, we see that the quotient complex
$\mathfrak{X}^{\bullet}(M)/\mathfrak{X}_Z^{\bullet}(M)$ is acyclic. In fact, one can show that this complex is
isomorphic to the complex computing the Poisson cohomology of the Poisson manifold $(Z,\pi_Z)$ with coefficients in
the normal bundle $\nu_Z$ of $Z$, endowed with the canonical representation (see \cite{Thesis}). Thus, we conclude
that $H^{\bullet}_{\pi_Z}(Z;\nu_Z)=0$. This is surprising, since the cohomology with trivial coefficients
$H^{\bullet}_{\pi_Z}(Z)$ never vanishes, and moreover, this cohomology is infinite dimensional when the
cosymplectic structure is proper, i.e.\ when the foliation is given by the fibers of a submersion to $S^1$.
\end{remark}

\section{Stability of the singular locus}
In this section we begin the proof of Theorem \ref{deformations1}. We prove that a Poisson structure $C^1$-close to a log symplectic structure $\pi$ is also log symplectic and moreover, its singular locus is diffeomorphic to that of $\pi$.

This is related to the problem of stability of Poisson submanifolds, which studies persistence of Poisson submanifolds under deformations of the Poisson bivector, and was treated in the case of symplectic leaves in \cite{Crainic.Stability}. Heuristically, the infinitesimal condition for stability of the Poisson submanifold $N$ of $(M,\pi)$ is that the quotient complex $\mathfrak{X}^{\bullet}(M)/\mathfrak{X}_N^{\bullet}(M)$ is acyclic in degree two, where $\mathfrak{X}^{\bullet}_N(M)$ denotes as before the space of multi-vector fields tangent to $N$. In the case of $N$ being a compact symplectic leaf, it is proved in \cite{Crainic.Stability} that the condition
\[H^2\left( \mathfrak{X}^{\bullet}(M) / \mathfrak{X}^{\bullet}_N(M)\right)=0 \]
ensures the stability of $N$ and moreover that the space of Poisson submanifolds nearby $N$ is parameterized by $H^1\left( \mathfrak{X}^{\bullet}(M) / \mathfrak{X}^{\bullet}_N(M)\right)$.

For a log symplectic structure $(M,\pi)$, we saw in the previous section that the complex $\mathfrak{X}^{\bullet}(M)/\mathfrak{X}_Z^{\bullet}(M)$ is acyclic in all degrees. This observation is the infinitesimal counterpart of the next Lemma.

\subsubsection*{Remark on the topologies.}
We use the $C^1$-topologies on $\Diff(M)$ and on $\Gamma(\wedge^2 TM)$, and the $C^0$-topology on
$\Gamma(\wedge^{2}T_ZM)$. Similarly, we endow the space $\Poiss(M)$, of all Poisson structures on $M$, with the
induced $C^1$-topology, and the space $\Log(M,Z)\subset \Gamma(\wedge^2 T_ZM)$, of all log symplectic
structures on $M$ with singular locus $Z$, with the induced $C^0$-topology.

The reader should be warned that the $C^0$-topology on $\Gamma(\wedge^{2}T_ZM)$ is not the subspace
$C^0$-topology induced from $\Gamma(\wedge^{2}TM)$. To see this, let $U$ be an open neighborhood around a point in $Z$ with
coordinates $(t,x)\in \mathbb{R}\times\mathbb{R}^{2n-1}$ such that $Z\cap U=\{t=0\}$. A bivector of the form
$f(t,x)\frac{\partial}{\partial t}\wedge \frac{\partial}{\partial x_i}$ is in $\Gamma(\wedge^2T_ZM)$ if and only if
$f(t,x)=tg(t,x)$ for some smooth function $g(t,x)$. Now the $C^0$-norm on $\Gamma(\wedge^2TM)$ computes the
supremum of $f(t,x)$, whereas the $C^0$-norm on $\Gamma(\wedge^2T_ZM)$ computes the supremum of $g(t,x)$.
This makes the following lemma more subtle than one would expect.

\begin{lemma}\label{sing_locus} Given a compact log symplectic manifold $(M,\pi)$ with singular locus $Z$, any $C^1$-close Poisson structure is log symplectic and has singular locus diffeomorphic to $Z$. More precisely, there is a $C^1$-open $\mathcal{V}\subset \Poiss(M)$ around $\pi$ and there is a map\break
$\Phi:\mathcal{V}\to \Diff(M)$ satisfying $\Phi(\pi)=\id_M$ and $\Phi(\pi')_{*}(\pi')\in \Log(M,Z)$ for every $\pi'\in
\mathcal{V}$. Moreover, the map
\[\Psi:\mathcal{V}\rmap \Log(M,Z), \qquad \Psi(\pi'):=\Phi(\pi')_{*}(\pi')\]
is continuous for the $C^1$-topology on $\mathcal{V}$ and the $C^0$-topology on $\Log(M,Z)$.
\end{lemma}

\begin{proof}
First we assume that $M$ is orientable. We prove this case in two steps: first we construct $\Phi$ and $\mathcal{V}$ and second we prove that $\Phi$ has the desired continuity property.\\

\textit{Step 1. Construction of $\Phi$ and $\mathcal{V}$.}

Since $M$ is orientable, the normal bundle to $Z$ is trivial, and we can find a tubular neighborhood $E\to M$, $E\cong \mathbb{R}\times Z$. Denote by $t$ the coordinate on $\mathbb{R}$. Now,
$\mu:=\frac{1}{t}\wedge^n \pi|_{E}$ is a nowhere vanishing section of $\wedge^{2n}TE$, and for every smooth
bivector $w$, we have that $\wedge^n w|_{E}=h_w(t,x)\mu$, for some smooth function $h_w$ on $E$. Moreover, the
assignment $w\mapsto h_w|_{[-2,2]\times Z}$ is continuous with respect to the $C^1$-topologies, and $h_{\pi}=t$.

Consider the $C^1$-open neighborhood $\mathcal{D}\subset C^{\infty}([-2,2]\times Z)$ around the $t$ consisting of functions $h$ such that ${\partial h(t,x)}/{\partial t}>0$ and $|h(t,x)-t|<1$. Then, for any $h\in\mathcal{D}$, we have that the function
\[\varphi_h: [-2,2]\times Z\rmap \mathbb{R}\times Z, \ \  (t,x)\mapsto (h(t,x),x),\]
is a diffeomorphism onto its image and that $\varphi_h\left((-2,2)\times Z\right)\supset[-1,1]\times Z$. Moreover, the
assignment $\varphi_h\mapsto \varphi_h^{-1}|_{[-1,1]\times Z}$ is continuous with respect to the $C^1$-topologies.

Let $\mathcal{V}'$ be the $C^1$-open neighborhood in $\Gamma(\wedge^2TM)$ consisting of elements $w$ that satisfy
\[\left(\wedge^n w\right)^{-1}(0)\subset (-2,2)\times Z \ \ \textrm{ and }\ \  h_w|_{[-2,2]\times Z}\in\mathcal{D}.\]
Clearly, $\pi\in \mathcal{V}'$. Write $\varphi_{h_w}^{-1}(t,x)=(g_w(t,x),x)$ and $g_w(x):=g_w(0,x)$. We have that
\[\left(\wedge^n w\right)^{-1}(0)=Z_{w}:=\{(g_w(x),x) : x\in Z\},\]
and since $\partial h_w/\partial t\neq 0$ on $[-2,2]\times Z$, $\wedge^n w$ is transverse to the zero-section. Thus, if
$\pi'\in \mathcal{V}'$ is Poisson, then $\pi'$ is log symplectic with singular locus $Z_{\pi'}$.

We fix a smooth compactly supported function $\chi:\mathbb{R}\to [0,1]$, such that $\chi(t)=1$ for $t\in [-2,2]$ and
$|\chi'(t)|<1/2$ for $t\in\mathbb{R}$. Consider the following diffeomorphism:
\[\Phi(w):M\diffto M,\  \Phi(w)|_{M\backslash E}=\id_{M\backslash E}, \ \Phi(w)|_{E}(t,x):=(t-\chi(t)g_w(x),x).\]
The conditions $|\chi'(t)|<1/2$ and $|g_w(x)|\leq 2$ imply that $pr_1\circ \partial \Phi(w)/\partial t>0$; and hence $\Phi(w)$ is 
indeed a diffeomorphism. The fact that $\chi(t)=1$ on $[-2,2]$ implies that $\Phi(w)$ maps $Z_{w}$ onto $Z$. Let $\mathcal{V}:=\mathcal{V}'\cap\Poiss(M)$.\\

\textit{Step 2. Continuity of $\Psi$.}

Clearly, the assignment $w\mapsto \Phi(w)$ is continuous for the $C^1$-topologies, and therefore the assignment
$w\mapsto \Phi(w)_*(w)$ is continuous for the $C^1$-topology on $\mathcal{V}'$ and the $C^0$-topology on
$\Gamma(\wedge^2TM)$. For a Poisson structure $\pi'\in\mathcal{V}'$, we have that $\Phi(\pi')_*(\pi')$ belongs to
$\Log(M,Z)$. Now, here comes the more subtle point of the proof: the fact that the $C^0$-topology on $\Log(M,Z)$ is
not the subspace topology induced from $\Gamma(\wedge^2TM)$ does not allow us to conclude yet the proof. Also, for
an arbitrary $w\in\mathcal{V}'$, $\Phi(w)_*(w)$ is not an element of $\Gamma(\wedge^2T_ZM)$; therefore we have
to restrict $\Phi$ to $\mathcal{V}=\mathcal{V}'\cap\Poiss(M)$.

Consider a finite open cover of $Z$ by coordinate charts $\{U_l\cong \mathbb{R}^{2n-1}\}_{l}$ such that the closed balls $\{\overline{B}_l\}_{l}$ of radius 1 still cover $Z$. In one of these charts $U_l$ with coordinates $(x_i)$, an element $\sigma\in\Gamma(\wedge^2T_ZM)$ can be written as
\[\sigma=\sum_i a^l_i(t,x) t \frac{\partial}{\partial t}\wedge\frac{\partial}{\partial x_i}+\sum_{i,j}b^l_{ij}(t,x)\frac{\partial}{\partial x_i}\wedge\frac{\partial}{\partial x_j}.\]
The $C^0$-norm of $\sigma|_{[-2,2]\times Z}$ (as an element in $\Gamma(\wedge^2T_ZM)$) is the supremum of the functions
$a^l_i$ and $b^l_{ij}$ on $[-2,2]\times \overline{B}_l$. Thus to compute the $C^0$-norm of $\sigma$ on $M$, one needs to measure its $C^0$-norm as an element in $\Gamma(\wedge^2TM)$ and the supremum norm of the coefficients $a^l_{i}$ in a small neighborhood of $Z$. In our
case, we know that the map $\mathcal{V}\ni \pi'\mapsto \Phi(\pi')_*(\pi')\in\Gamma(\wedge^2 TM)$ is continuous for
the $C^0$-topology on the second space, thus it suffices to check that the corresponding functions ``$a^l_i$'' also vary
continuously.

Consider $\pi'\in \mathcal{V}$, and denote $g:=g_{\pi'}$. Denote the local expression of $\pi'$ on $U_l$ by
\[\pi':=\sum_i A^l_i(t,x)\frac{\partial}{\partial t}\wedge\frac{\partial}{\partial x_i}+\frac{1}{2}\sum_{i,j}B^l_{ij}(t,x)\frac{\partial}{\partial x_i}\wedge\frac{\partial}{\partial x_j}.\]
The fact that $\pi'$ is tangent to $Z_g$ is written in coordinates as follows:
\begin{equation}\label{tang}
\sum_i A^l_i(g(x),x)\bigg(\frac{\partial}{\partial x_i}+\frac{\partial g}{\partial x_i}(x)\frac{\partial}{\partial t}\bigg)+\sum_{i,j}B^l_{ij}(g(x),x)\frac{\partial g}{\partial x_i}(x)\frac{\partial}{\partial x_j}=0.
\end{equation}
Now, $\Phi(\pi')|_{[-2,2]\times Z}$ is of the form $(t,x)\mapsto (t-g(x),x)$, and therefore
\begin{align*}
\Phi(\pi')_*(\pi')&|_{[-1,1]\times Z}=\frac{1}{2}\sum_{i,j}B^l_{ij}(t+g(x),x)\frac{\partial}{\partial x_i}\wedge\frac{\partial}{\partial x_j}+\\
&+\sum_i\bigg( A^l_i(t+g(x),x)-\sum_jB^l_{ij}(t+g(x),x)\frac{\partial g}{\partial x_j}(x)\bigg)\frac{\partial}{\partial t}\wedge\frac{\partial}{\partial x_i}.
\end{align*}
Equation (\ref{tang}) implies that the coefficient of $\frac{\partial}{\partial
t}\wedge\frac{\partial}{\partial x_i}$ vanishes at $t=0$ and therefore we get that the coefficient of $t\frac{\partial}{\partial
t}\wedge\frac{\partial}{\partial x_i}$ is equal to
\begin{align*}
a^l_i(t,x):=&\frac{1}{t}\bigg( A^l_i(t+g(x),x)-\sum_jB^l_{ij}(t+g(x),x)\frac{\partial g}{\partial x_j}(x)\bigg)=\\
&=\int_{0}^1\bigg(\frac{\partial A^l_i}{\partial t}(st+g(x),x)-\sum_j\frac{\partial B^l_{ij}}{\partial
t}(st+g(x),x)\frac{\partial g}{\partial x_j}(x)\bigg) ds.
\end{align*}
We see now that this explicit formula, which holds on $[-1,1]\times Z$, implies that these coefficients depend continuously on $\pi'$, i.e., the map $\pi'\mapsto (a^l_i)$ is continuous with respect to the $C^1$-topology on $\mathcal{V}$ and the
$C^0$-topology on $C^{\infty}([-1,1]\times Z)$. Indeed, let $\pi''$ be a Poisson structure $C^1$-close to $\pi'$.
Denote the coefficients of $\pi''$ on $U_l$ by $\tilde{A}^l_i$, $\tilde{B}^l_{ij}$. Similarly, for $\pi''$ we get a function $\tilde{g}$ taking values in $[-2,2]$ which, by the previous argument, is $C^1$-close to $g$. We have then
\[\tilde{A}^l_i=A^l_i+\alpha^l_i,\quad \tilde{B}^l_{ij}=B^l_{ij}+\beta^l_{ij},\quad \tilde{g}=g+\delta,\]
for some functions $\alpha^l_i$, $\beta^l_{ij}$, $\delta$ which are $C^1$-small. We need to compute the $C^0$-norm of $\tilde{a}^l_i-a^l_i$ on $[-1,1]\times \overline{B}_l$, for which we use the formula
\begin{align*}
(\tilde{a}^l_i-&a^l_i)(t,x)=\int_0^1\bigg(\int_0^1 \frac{\partial^2 A^l_i}{\partial t^2}\Big(\delta(x)s'+st+g(x),x\Big)\delta(x)ds' + \frac{\partial \alpha^l_i}{\partial t}(st+\tilde{g}(x),x)\bigg)ds- \\
& -  \sum_j \Big(\frac{\partial g}{\partial x_j} +\frac{\partial \delta}{\partial x_j}\Big) \int_0^1 \Big(\int_0^1 \frac{\partial^2 B^l_{ij}}{\partial t^2}\big(\delta(x)s'+st+g(x),x\big)\delta(x)ds' + \frac{\partial \beta^l_{ij}}{\partial t}(st+\tilde{g}(x),x) \Big) ds
\end{align*}
Note that the $C^1$-norms of $ \alpha^l_i$ and $\beta^l_{ij}$ on $[-3,3]\times \overline{B}_l$ are bounded by a multiple of $||\pi''-\pi'||_{C^1}$. Now, since $st+\tilde{g}(x)\in [-3,3]$ for $s,t\in[0,1]$ and $x\in Z$, there is a constant $C>0$ such that
\begin{align*}
||\tilde{a}^l_i-a^l_i||_{C^0} \leq C\big(||\delta||_{C^0}+||\pi''-\pi'||_{C^1}\big).
\end{align*}
This implies the desired continuity and finishes the proof in the case when $M$ is orientable.\\

\textit{Step 3. The non-orientable case.}

If $M$ is not orientable, consider the orientable double cover $p:\widetilde{M}\to M$, and let
$\gamma:\widetilde{M}\diffto \widetilde{M}$ be the corresponding deck transformation. Then
$\widetilde{\pi}:=p^*(\pi)$ is a $\gamma$-invariant log symplectic structure on $\widetilde{M}$ with singular locus
$\widetilde{Z}:=p^{-1}(Z)$. Consider a $\gamma$-invariant tubular neighborhood of $\widetilde{E}\to
\widetilde{M}$, which admits a trivialization $\widetilde{E}\diffto \mathbb{R}\times \widetilde{Z}$ on which
$\gamma$ acts by $\gamma(t,x):=(-t,\gamma(x))$ (for more details on this construction, see \cite{MO}). Let
$\widetilde{\mathcal{V}}$ be the $C^1$-neighborhood of $\widetilde{\pi}$ constructed in the first part, and let
$\widetilde{\Phi}:\widetilde{\mathcal{V}}\to \Diff(\widetilde{M})$ be the corresponding map. Let $\mathcal{V}$ be
the set of Poisson structures $\pi'$ such that $p^*(\pi')\in\widetilde{\mathcal{V}}$. Clearly, $\mathcal{V}$ is a
$C^1$-neighborhood of $\pi$. Note that the construction of $\widetilde{\Phi}$ is such that
$\widetilde{\Phi}(p^*(\pi'))$ is $\gamma$-equivariant: we have that $\gamma^*(\mu)=-\mu$, thus $h_{p^*(\pi')}\circ
\gamma(t,x)=-h_{p^*(\pi')}(t,x)$, thus also $g_{p^*(\pi')}\circ \gamma(x)=-g_{p^*(\pi')}(x)$; and therefore, by
choosing $\chi(t)$ to be an even function, we obtain that $\widetilde{\Phi}(p^*(\pi'))$ is $\gamma$-equivariant. This
implies that $\widetilde{\Phi}(p^*(\pi'))$ covers a diffeomorphism which we denote by $\Phi(\pi')\in\Diff(M)$ and which
satisfies
\[\widetilde{\Phi}(p^*(\pi'))_*(p^*(\pi'))=p^*\left(\Phi(\pi')_*(\pi')\right).\]
Since $\widetilde{\Phi}(p^*(\pi'))_*(p^*(\pi'))\in \Log(\widetilde{M},\widetilde{Z})$, it follows that
$\Phi(\pi')_*(\pi')\in\Log(M,Z)$. Moreover, since the pull-back maps $p^*:\mathcal{V}\to \widetilde{\mathcal{V}}$
and $p^*:\Log(M,Z)\to \Log(\widetilde{M},\widetilde{Z})$ are embeddings for the $C^1$-topology and for the $C^0$-topology, respectively, it follows that the map $\pi'\mapsto \Phi(\pi')_*(\pi')$ is continuous. This concludes the proof.
\end{proof}

\section{Proof of Theorem \ref{deformations1}}\label{pruf}
In this section we prove part (b) and (c) of Theorem \ref{deformations1}. By the result in the previous section (Lemma \ref{sing_locus}), all small deformations can be represented by nearby log symplectic structures which have the same singular locus. In this setup, the b-version of the Moser argument (Lemma \ref{Moser}) is used to prove that log symplectic structures are
determined up to diffeomorphism by their b-cohomology. This will conclude the proof of part (b) of Theorem
\ref{deformations1}. Finally, part (c) is proved by showing that small diffeomorphisms do not change the b-cohomology class.

Fix $(M^{2n},\pi)$ a compact log symplectic manifold with singular locus $Z$. Let $\omega:=\pi^{-1}\in \Omega^2_Z(M)$ denote the inverse of $\pi$. Consider, as in the statement of Theorem \ref{deformations1}, $\varpi_1,\ldots,\varpi_l$ closed two-forms on $M$ and $\gamma_1,\ldots,\gamma_k$ closed
one-forms on $Z$ such that their cohomology classes form a basis of $H^2(M)$ and $H^1(Z)$ respectively. For $\epsilon\in\mathbb{R}^l$ and $\delta\in \mathbb{R}^k$, denote by $\varpi_{\epsilon}:=\sum_{i=1}^l\epsilon_i\varpi_i$ and $\gamma_{\delta}:=\sum_{i=1}^k\delta_i\gamma_i$. Fix a
tubular neighborhood $E$ of $Z$, with projection $p:E\to Z$, and let $\lambda$ be a distance function for $E$. Denote by
\[\omega_{\epsilon,\delta}:=\omega+\varpi_{\epsilon}+d\log(\lambda)\wedge p^*(\gamma_{\delta}).\]
The Mazzeo-Melrose decomposition implies that the map $\mathbb{R}^{l+k}\to H^2_Z(M)$, $(\epsilon,\delta)\mapsto [\omega_{\epsilon,\delta}]$ is a linear isomorphism.

\vspace{0.2cm}

Next, we construct a convex neighborhood consisting of non-degenerate b-forms, which will be used in the proof of part (b).

\begin{lemma}
There is a $C^0$-open neighborhood $\mathcal{W}$ around $\omega$ in the space of all closed b-two-forms
$\Omega^2_{Z,closed}(M)$, such that every $\omega'\in \mathcal{W}$ is nondegenerate, and moreover, if
$(\epsilon,\delta)\in\mathbb{R}^{l+k}$ is such that $[\omega']=[\omega_{\epsilon,\delta}]$, then the entire path
$(1-t)\omega'+t\omega_{\epsilon,\delta}$ for $t\in [0,1]$ is contained in $\mathcal{W}$.
\end{lemma}
\begin{proof}
Fix a metric on $T_ZM$, and consider the following map on $\Omega^2_Z(M)$:
\[\omega'\in \Omega^2_Z(M), \ \omega'\mapsto ||\omega'||:=\sup\{|\pi^{\sharp}\circ \omega'^{\sharp}(V)-V| : V\in T_ZM, |V|=1\}.\]
Clearly, $||\omega||=0$, thus the following is a $C^0$-neighborhood of $\omega$:
\[\widetilde{\mathcal{W}}:=\{\omega'\in\Omega^2_Z(M) : d\omega'=0,\ ||\omega'||<1\}\subset \Omega^2_{Z,closed}(M).\]
Note that $\widetilde{\mathcal{W}}$ is a convex neighborhood consisting entirely of log symplectic structures on
$(M,Z)$.
Let $U\subset\mathbb{R}^{l+k}$ be the space consisting of pairs $(\epsilon,\delta)$ such that
$\omega_{\epsilon,\delta}\in \widetilde{\mathcal{W}}$. Clearly, $U$ is a convex open neighborhood of $0$. Next, note that taking the cohomology class of a closed b-form $\omega'\mapsto [\omega']$ is continuous for the
$C^0$-topology on the space of all closed b-forms. To see this, observe that the decomposition form Section \ref{b-geom}
\[\Omega_{Z,closed}^{\bullet}(M)\diffto \Omega_{closed}^{\bullet}(M)\oplus \Omega_{closed}^{\bullet-1}(Z)\]
is $C^0$-continuous, and that taking de Rham cohomology is also $C^0$-continuous, because de Rham cohomology
can be detected by integrating along compact submanifolds, and integration is $C^0$-continuous. This shows that the
following is $C^0$-open.
\[\mathcal{W}:=\{\omega'\in\widetilde{\mathcal{W}} : [\omega']=[\omega_{\epsilon,\delta}] \textrm{ for some }(\epsilon,\delta)\in U\}.\]
Clearly, $\mathcal{W}$ has the required properties.
\end{proof}

\begin{proof}[of Theorem \ref{deformations1}, part (b)]
Consider the $C^0$-neighborhood $\mathcal{W}$ from the previous lemma. Consider also the $C^1$-open neighborhood $\mathcal{V}\subset\Poiss(M)$, the map $\Phi:\mathcal{V}\to \Diff(M)$, and the map $\Psi:\mathcal{V}\to \Log(M,Z)$,
$\Psi(\pi'):=\Phi(\pi')_*(\pi')$ from Lemma \ref{sing_locus}. The inversion map $\Log(M,Z)\to \Omega^2_Z(M)$ is
continuous for the $C^0$-topologies, and by the continuity of $\Psi$, the following is a $C^1$-open neighborhood around $\pi$:
\[\mathcal{U}:=\{\pi'\in\mathcal{V} : \Psi(\pi')^{-1}\in \mathcal{W}\}\subset \Poiss(M).\]
Let $\pi'\in\mathcal{U}$. Then $\pi'$ is diffeomorphic to the log symplectic structure
$(\omega')^{-1}:=\Psi(\pi')\in\Log(M,Z)$ via the map $\Phi(\pi')$.
Let $(\epsilon,\delta)\in \R^{l+k}$ be such that $[\omega']=[\omega_{\epsilon,\delta}]$. Since $\omega'\in\mathcal{W}$, the Moser argument from Lemma \ref{Moser} implies that $\omega'$ is diffeomorphic to $\omega_{\epsilon,\delta}$.
\end{proof}

To prove part (c), first we show the existence of path connected neighborhoods of $\mathrm{id}_M$ in the space of diffeomorphisms that fix $Z$.

\begin{lemma}
Let $(M,Z)$ be a compact b-manifold. Then there exists a $C^1$-open neighborhood $\mathcal{D}$ around $\id_M$ in $\Diff(M)$, such that, for every
$\varphi\in\mathcal{D}$ that leaves $Z$ invariant, there is a smooth isotopy $\varphi_t$, for $t\in [0,1]$, with
$\varphi_0=\id_M$ and $\varphi_1=\varphi$, consisting of diffeomorphisms that send $Z$ to $Z$.
\end{lemma}
\begin{proof}
First, recall the standard construction of a $C^1$-neighborhood of $\id_M$ in $\Diff(M)$. Let
$g$ be a metric on $M$, with exponential map $\exp:TM\to M$. There is a $C^1$-open neighborhood $\mathcal{A}\subset
\Gamma(TM)$ around the zero-section, such that for every $V\in\mathcal{A}$, the induced map
\[\exp_*(V)\in C^{\infty}(M,M), \ \ \exp_*(V)(x):=\exp(V_x)\]
is a diffeomorphism. Moreover, $\exp_*$ is a homeomorphism for the $C^1$-topologies onto a $C^1$-open neighborhood
$\mathcal{D}\subset \Diff(M)$ around $\id_M$.

Consider a metric for which $Z$ is geodesically closed. We claim that for such a metric, if $\mathcal{A}$ is small enough then the set $\mathcal{D}$ has the required property.
Let $\varphi\in\mathcal{D}$ be such that $\varphi$ sends $Z$ to $Z$, and write $\varphi=\exp_*(V)$, for some $V\in
\mathcal{A}$. We consider the isotopy $\varphi_t:=\exp_*(tV)$ for $t\in [0,1]$. For $x\in Z$, the curve
$\varphi_t(x)=\exp(tV_x)$ is a geodesic that starts and ends in $Z$. For small enough $\mathcal{A}$, we may assume
that $V_x$ is shorter than the injectivity radius of $\exp$. Since $Z$ is geodesically closed, this implies that $V_x\in
T_xZ$, and so $\varphi_t(x)\in Z$ for all $t$.
\end{proof}

\begin{proof}[of Theorem \ref{deformations1}, part (c)]
Consider the open neighborhood $\mathcal{D}$ described in the previous lemma. Let $\pi_0,\pi_1\in \Log(M,Z)$ be two log symplectic structures such that there is some $\varphi\in\mathcal{D}$ for which $\varphi_*(\pi_0)=\pi_1$. Denote by $\omega_0:=\pi_0^{-1}$ and $\omega_1:=\pi_1^{-1}$. We claim that
$[\omega_0]=[\omega_1]$. Since $\varphi$ sends $Z$ to $Z$, there is an isotopy $\varphi_t$ with the
properties from the lemma, i.e.\ $\varphi_t$ is a b-diffeomorphism for all $t\in [0,1]$. Denote the generating time-dependent
b-vector field by $X_t$,
\[X_t(x)=\frac{d\varphi_t}{dt}\left(\varphi_t^{-1}(x)\right).\]
We apply now the ``reverse Moser trick''. Since $\varphi_1^*(\omega_1)=\omega_0$, we have that
\[\omega_0-\omega_1=\int_{0}^1\frac{d}{dt}\left(\varphi_t^*(\omega_1)\right)dt=\int_{0}^1\varphi_t^*(L_{X_t}\omega_1)dt=d\int_0^1\varphi_{t}^*(\iota_{X_t}\omega_1)dt.\]
Hence $[\omega_0]=[\omega_1]\in H^2_Z(M)$, and this proves part (c) of Theorem \ref{deformations1}.
\end{proof}

\section{Examples}\label{exam}

In this section we present some simple examples to illustrate some phenomena that appear in deforming log symplectic structures.

\begin{example}\rm
Consider the unit sphere $S^2$ embedded in $\mathbb{R}^3=\{(x,y,z)\}$. We use the angle
$\theta:=\tan^{-1}(y/x)\in S^1$ and $z\in (-1,1)$ as coordinates on $S^2$ minus the two poles. The bivector
$\partial_{\theta}\wedge\partial_{z}$ extends to a nondegenerate Poisson structure on $S^2$. Consider the log symplectic
structure
\[\pi:=z\partial_{\theta}\wedge\partial_z.\]
The singular locus of $\pi$ is the equator $S^1:=\{z=0\}\cap S^2$, and the induced 1-form on $S^1$ is
$d\theta|_{S^1}$. The corresponding b-two-form is
\[\omega:=\frac{d{z}}{z}\wedge d{\theta}.\]
The generators of the b-cohomology group $H^2_{S^1}(S^2)$ are $[\omega]$ and $[dz\wedge d\theta]$. Under the
canonical decomposition $H^2_{S^1}(S^2)\cong H^2(S^2)\oplus H^1(S^1)$, $[\omega]$ corresponds to the
generator of $H^1(S^1)$ and $[dz\wedge d\theta]$ to the generator of $H^2(S^2)$. Every Poisson structure
$C^1$-near $\pi$ is isomorphic to one of the form
\[\pi_{\epsilon,\delta}:=\left(\omega+\epsilon dz\wedge d\theta+\delta \omega\right)^{-1}=\frac{z}{1+\epsilon z+\delta}\partial_{\theta}\wedge\partial_z.\]
Of course, this example fits in the classification of log symplectic structures on compact orientable surfaces
from \cite{Radko.2002}.
\end{example}

\begin{example}\rm
The log symplectic structure $\pi=z\partial_{\theta}\wedge\partial_z$ on $S^2$ is invariant under the symmetry
$\gamma(x,y,z):=(-x,-y,-z)$. Let $p:S^2\to\mathbb{P}^2:=S^2/\{\id,\gamma\}$ be the projection onto the real
projective plane. Invariance of $\pi$ implies that $p_*(\pi)$ is a log symplectic structure on $\mathbb{P}^2$ with singular locus
$p(S^1)=S^1/\{\id,\gamma\}\cong S^1$. Since $H^2(\mathbb{P}^2)=0$, we have that the b-cohomology is
concentrated in $H^1(p(S^1))\cong \mathbb{R}$. The corresponding 1-parameter family of deformations is
\[p_*(\pi_{0,\delta})=1/(1+\delta)p_*(\pi).\]
\end{example}

\begin{example}\rm
In the symplectic world, the analogous statement to Theorem \ref{deformations1} allows us to find a
$C^0$-neighborhood of the symplectic structure in which symplectic structures are classified by $H^2(M)$. In the
log symplectic case, the neighborhood in $\Poiss(M)$ has to be $C^1$, merely because transversality is a $C^1$
condition. Actually, for the proof of Lemma \ref{sing_locus}, one needs $C^1$-closeness only in a small neighborhood
around the singular locus, while outside one can consider $C^0$-open neighborhoods.

To illustrate this phenomenon, consider again the log symplectic structure $\pi=z\partial_{\theta}\wedge\partial_z$ on $S^2$.
We look at families of Poisson structures of the form
\[\pi_{\varepsilon}:=h_{\varepsilon}(z)z\partial_{\theta}\wedge\partial_z,\]
for small $\varepsilon>0$, where
$h_{\varepsilon}:\mathbb{R}\to [-1,1]$ are functions such that $h_{\varepsilon}(z)=1$ for $|z|\geq \varepsilon$. Note
that the $C^0$-distance between $\pi_{\varepsilon}$ and $\pi$ is less than $2\varepsilon$.

First, consider $h_{\varepsilon}$ such that it vanishes on $[-\epsilon/2,\epsilon/2]$. We see that $\pi$ can be
$C^0$-approximated by Poisson structures that are not log symplectic.

Consider now $h_{\varepsilon}$ such that it vanishes linearly at $\pm\varepsilon/2$, and only at these points. The
resulting structures are log symplectic with singular locus the three circles $z=-\varepsilon/2$, $z=0$, $z=\varepsilon/2$. This
shows that $\pi$ can be $C^0$-approximated by log symplectic structures whose singular locus is not diffeomorphic to that of
$\pi$.
\end{example}

\begin{example}\rm
Consider $S^2\times T^2$ endowed with the product log symplectic structure
\[\pi:=z\partial_{\theta}\wedge\partial_z+\partial_{\theta_2}\wedge\partial_{\theta_1},\]
where $\theta_1 ,\theta_2$ are the coordinates on the two-torus $T^2= S^1\times S^1$. The singular locus of $\pi$ is
the three-torus $S^1\times T^2$, where the first $S^1$ denotes the equator in $S^2$. The induced cosymplectic
structure on $S^1\times T^2$ is $(d\theta_1\wedge d\theta_2,d\theta)$. The b-two-form is
\[\omega=d\log|z|\wedge d{\theta}+d\theta_1\wedge d\theta_2.\]
The second b-cohomology $H^2_{S^1\times T^2}(S^2\times T^2)$ is spanned by
\[[dz\wedge d\theta], \ [d\theta_1\wedge d\theta_2],\ [d\log|z|\wedge d\theta],\ [d\log|z|\wedge d\theta_1],\ [d\log|z|\wedge d\theta_2],\]
where the first two generators correspond to $H^2(S^2\times T^2)$ and the last three to $H^1(S^1\times T^2)$.
By Theorem \ref{deformations1}, we conclude that any Poisson structure $C^1$-close to $\pi$ is diffeomorphic to a
log symplectic structure on $(S^2\times T^2, S^1\times T^2)$ with b-two-form
\[\omega'=\epsilon_1dz\wedge d\theta+ (1+\epsilon_2)d\theta_1\wedge d\theta_2+d\log|z|\wedge \left( (1+\delta_1)d{\theta}+\delta_2 d\theta_1+\delta_3d\theta_2\right),\]
for some $\epsilon_1,\epsilon_2,\delta_1,\delta_2,\delta_3\in \mathbb{R}$ small enough. The cosymplectic structure is
\[\left((1+\epsilon_2)d\theta_1\wedge d\theta_2, (1+\delta_1)d{\theta}+\delta_2 d\theta_1+\delta_3d\theta_2\right).\]
In particular, the foliation is a ``generalized Kronecker foliation'', i.e.\ its pullback to $\mathbb{R}^3$ via the
projection $\mathbb{R}^3\to S^1\times T^2$ is a foliation by the parallel affine planes
\[\{(x_1,x_2,x_3)\in\mathbb{R}^3: (1+\delta_1)x_1+\delta_2 x_2+\delta_3x_3=C\}_{C\in\mathbb{R}}.\]
This is quite remarkable, since the foliation $\ker(d\theta)$ can be $C^{\infty}$-approximated by foliations which
behave very differently, e.g.\ the family of foliations $\mathcal{F}_{\epsilon}$ given by the kernel of
$d\theta+\epsilon\chi(\theta)d\theta_1$, where $\chi$ is a function such that $\chi(\theta)=\theta$ for small $|\theta|$.
\end{example}


\section*{Appendix}

\begin{proof}[of lemma \ref{functorial1}]
Recall from \cite{David} the definition of the \emph{regularized volume}: namely, let $\omega$ be a top b-form
on a compact oriented b-manifold $(N,W)$. Its regularized volume is defined by the limit:
\[\mathrm{Vol}_N(\omega):=\lim_{\epsilon\to 0}\int_{\lambda\geq \epsilon}\omega,\]
where $\lambda:N\backslash W\to (0,\infty)$ is a distance function. This limit exists for any $\lambda$, and does not dependent on the choice of $\lambda$.

First we show that the map $\sigma:\Omega^{k-1}(Z)\to \Omega^k_Z(M)$ maps closed forms to forms that satisfy the condition from the lemma. For this, consider a b-map $i:(N,W)\to (M,Z)$ from a compact oriented manifold $N$ of dimension $k$, and consider a closed $k-1$-form $\mu$ on $Z$. Since $i$ is a b-map, note that $\lambda\circ i$ is a distance function for $(N,W)$; thus the volume can be computed by:
\[\mathrm{Vol}_N(i^*(\sigma(\mu))):=\lim_{\epsilon\to 0}\int_{N_{\epsilon}}i^*\Big(d\big(\log(\lambda)p^*(\mu) \big)\Big)=\lim_{\epsilon\to 0}\int_{N_{\epsilon}}d\big(\log(\lambda\circ i)(p\circ i)^*(\mu) \big),\]
where $N_{\epsilon}:=\{x\in N: \lambda(i(x))\geq \epsilon\}$. For $\epsilon$ small enough, we will prove that:
\begin{equation}\label{integral}
\int_{N_{\epsilon}}d\big(\log(\lambda\circ i)(p\circ i)^*(\mu) \big)=0.
\end{equation}
For small enough $\epsilon$, $N_{\epsilon}$ is a manifold with boundary $W_{\epsilon}:=\{x\in N:
\lambda(i(x))=\epsilon\}$, and by Stokes' theorem, the integral in (\ref{integral}) equals
$\log(\epsilon)\int_{W_{\epsilon}}(p\circ i)^*(\mu)$. Note that $W_{\epsilon}$, with the opposite orientation, is also the boundary of the manifold $E_{\epsilon}:=\{x\in N: \lambda(i(x))\leq \epsilon\}$. Therefore, again using Stokes' theorem we obtain 
\[\int_{W_{\epsilon}}(p\circ i)^*(\mu)=-\int_{E_{\epsilon}}d(p\circ i)^*(\mu)=-\int_{E_{\epsilon}}(p\circ i)^*(d\mu)=0,\]
which proves that (\ref{integral}) holds. 

Next, we claim that for any b-map $i:(N,W)\to (M,Z)$ from a $k$-dimensional, compact oriented b-manifold $(N,W)$,
and any exact $k$-b-form $d\eta$, we have that $\mathrm{Vol}_N(i^*(d\eta))=0$. Write $\eta=\alpha+\sigma(\beta)$,
for $\alpha\in \Omega^{k-1}(M)$ and $\beta\in\Omega^{k-2}(Z)$. Then
\[\mathrm{Vol}_N(i^*(\omega))=\mathrm{Vol}_N(di^*(\alpha))+\mathrm{Vol}_N(\sigma(d\beta)).\]
The first term vanishes, since it is the integral of an exact de-Rham form, and the second vanishes by the above. Thus,
we conclude that the b-map $i:(N,W)\to (M,Z)$ induces a map
\[\mathrm{Vol}_N:H^k_Z(M)\rmap \mathbb{R},\]
which vanishes on $\sigma(H^{k-1}(Z))$.

Let $C^k_Z(M)\subset H^{k}_Z(M)$ be the space of b-cohomology classes
$[\omega]\in H^{k}_Z(M)$ satisfying $\mathrm{Vol}_N(i^*(\omega))=0$ for all b-maps $i:(N,W)\to (M,Z)$ where
$(N,W)$ is a compact oriented $k$-dimensional b-manifold. By the above, we have that $\sigma(H^{k-1}(Z))\subset C^{k}_Z(M)$. Now, $H_k(M,\mathbb{R})$ is generated by smooth oriented submanifolds $i:N\to M$, which can be assumed to be transverse to $Z$, hence yielding b-maps
$i:(N,i^{-1}(Z))\to (M,Z)$. This implies that $C^k_Z(M)\cap H^k(M)=\{0\}$. Thus, using the decomposition (\ref{Mazzeo}), this shows that
$C^k_Z(M)=\sigma(H^{k-1}(Z))$.
\end{proof}

\bibliographystyle{amsplain}
\def\lllll{}

\end{document}